\documentclass{amsart}
\usepackage{amsfonts}

\usepackage{amscd}
\usepackage{thmdefs}

\input{tcilatex}
\theoremstyle{definition}
\theoremstyle{remark}
\numberwithin{equation}{section}

\input tcilatex

\begin{document}
\title[Quasiconvexity and density topology]{Quasiconvexity and density topology}
\author{Patrick J. Rabier}
\address{Department of mathematics, University of Pittsburgh, Pittsburgh, PA 15260}
\email{rabier@imap.pitt.edu}
\subjclass{52A41, 26B05}
\keywords{Density topology, quasiconvex function, approximate continuity, point of
continuity}
\maketitle

\begin{abstract}
We prove that if $f:\Bbb{R}^{N}\rightarrow \overline{\Bbb{R}}$ is
quasiconvex and $U\subset \Bbb{R}^{N}$ is open in the density topology, then 
$\sup_{U}f=\limfunc{ess}\sup_{U}f,$ while $\inf_{U}f=\limfunc{ess}\inf_{U}f$
if and only if the equality holds when $U=\Bbb{R}^{N}.$ The first (second)
property is typical of lsc (usc) functions and, even when $U$ is an ordinary
open subset, there seems to be no record that they both hold for all
quasiconvex functions.

This property ensures that the pointwise extrema of $f$ on any nonempty
density open subset can be arbitrarily closely approximated by values of $f$
achieved on ``large'' subsets, which may be of relevance in a variety of
issues. To support this claim, we use it to characterize the common points
of continuity, or approximate continuity, of two quasiconvex functions that
coincide away from a set of measure zero.
\end{abstract}

\section{Introduction\label{intro}}

To begin with matters of terminology, a quasiconvex function $f$ on $\Bbb{R}%
^{N}$ refers to an extended real-valued function whose lower level sets $%
\{x\in X:f(x)<\alpha \}$ are convex for every $\alpha \in \Bbb{R}.$ The same
class is obtained if the level sets $\{x\in X:f(x)\leq \alpha \}$ are used
instead. These functions were first introduced by\footnote{%
The occasional claim that they were already investigated by von Neumann in
1928 is a gross exaggeration; see the historical article \cite{GuMo04}.} de
Finetti \cite{Fi49} in 1949, although the nomenclature was only coined by
Fenchel \cite{Fe53} a few years later.

A \emph{null set} is a subset of $\Bbb{R}^{N}$ of Lebesgue measure $0$ and
Lebesgue measure, simply called measure, is denoted by $\mu _{N}.$ Without
accompanying epithet, the words ``open'', ``interior'', ``closure'',
``boundary'', etc. and related symbols always refer to the euclidean
topology of $\Bbb{R}^{N}.$

Recall also that the density topology on $\Bbb{R}^{N}$ is the topology whose
open subsets are $\emptyset $ and the measurable subsets of $\Bbb{R}^{N}$
with density $1$ at each point. They will henceforth be referred to as \emph{%
density }open. Every open subset is density open. The (extended) real-valued
functions on $\Bbb{R}^{N}$ which are (semi)continuous when $\Bbb{R}^{N}$ is
equipped with the density topology and $\Bbb{R}$ with the euclidean topology
are the so-called \emph{approximately} (semi)continuous functions.

We shall only use elementary properties of the density topology. For
convenience, a brief summary is given in the next section.

If $f:\Bbb{R}^{N}\rightarrow \overline{\Bbb{R}}:=[-\infty ,\infty ]$ and $%
\alpha \in \Bbb{R},$ we set 
\begin{equation}
F_{\alpha }:=\{x\in \Bbb{R}^{N}:f(x)<\alpha \}.  \label{1}
\end{equation}
This will only be used without further mention when the function of interest
is called $f,$ so no ambiguity will arise.

Now, if $f$ is upper semicontinuous (usc for short) and $U\subset \Bbb{R}
^{N} $ is an open subset, it is trivial that 
\begin{equation}
\inf_{U}f=\limfunc{ess}\inf_{U}f.  \label{2}
\end{equation}
Indeed, since the lower level sets $F_{\alpha }$ are open, the intersection $%
U\cap F_{\alpha }$ has positive measure whenever it is nonempty. More
generally, (\ref{2}) is true and equally straightforward if $U$ is density
open and $f$ is approximately usc, but it fails if $U$ has only positive
measure, even if $f$ is finite and continuous or has any amount of extra
regularity.

Thus, heuristically at least, (\ref{2}) for every density open subset $%
U\subset \Bbb{R}^{N}$ is best possible for any measurable function $f.$ This
property, which ensures that $\inf_{U}f$ can be arbitrarily closely
approximated by values of $f$ achieved on ``large''subsets, is of possible
relevance in a variety of technical issues. It may fail to hold if the
function is modified at a single point, but elementary one-dimensional
examples show that it is more general than upper semicontinuity, even
approximate.

Likewise, if $f$ is approximately lower semicontinuous, then 
\begin{equation}
\sup_{U}f=\limfunc{ess}\sup_{U}f,  \label{3}
\end{equation}
for every density open subset $U$ of $\Bbb{R}^{N}$

The main result of this note (Theorem \ref{th3}) is that if $f$ is
quasiconvex, (\ref{3}) always holds and (\ref{2}) holds if and only if it
holds when $U=\Bbb{R}^{N}$ (Theorem \ref{th3}). Of course, (\ref{2}) and (%
\ref{3}) are trivial when $f$ is approximately continuous (in particular,
when $U$ is open and $f$ is continuous), but it is more surprising that they
continue to hold when $f$ is quasiconvex, without any continuity-like
requirement. (Needless to say, quasiconvexity does not imply approximate
continuity.) When $f$ is an arbitrary convex function -not necessarily
proper- an equivalent statement is given in Corollary \ref{cor4}.

In spite of the by now substantial literature involving quasiconvex
functions, this arguably notable property seems to have remained unnoticed,
even when $U$ is an euclidean open subset. At any rate, prior connections
between quasiconvexity in the sense of de Finetti and the density topology
(or approximate continuity) appear to be inexistent.

In Section \ref{equivalent}, we use (\ref{2}) and (\ref{3}) to compare the
points of (approximate) continuity of two real-valued quasiconvex functions $%
f$ and $g$ on $\Bbb{R}^{N}$ such that $f=g$ a.e., so that $f$ and $g$ have
the same essential infimum $m:=\limfunc{ess}\inf_{\Bbb{R}^{N}}f=$ $\limfunc{
ess}\inf_{\Bbb{R}^{N}}g.$

By a well known result of Crouzeix \cite{Cr81} (see also \cite{ChCr87}),
every real-valued quasiconvex function is Fr\'{e}chet differentiable a.e.
and so continuous a.e. A sharper property is even proved in Borwein and Wang 
\cite{BoWa05} in the lsc case. Thus, $f$ and $g$ above are simultaneously
continuous at the points of a large set, but this does not say whether $f$
is continuous at a given point $x$ where $g$ is known to be continuous.

In Theorem \ref{th6}, we show that this question and the same question for
points of approximate continuity can be given simple, yet complete answers:
A point $x$ of approximate continuity of $g$ is not a point of approximate
continuity of $f$ if and only if $m>-\infty ,$ $g(x)=m$ and $x\in F_{m}$
(see (\ref{1})) while a point $x$ of continuity of $g$ is not a point of
continuity of $f$ if and only if $m>-\infty ,g(x)=m$ and $x\in \cup _{\alpha
<m}\overline{F}_{\alpha }.$ The similarity and the difference between these
two results are better appreciated if it is noticed that $F_{m}=\cup
_{\alpha <m}F_{\alpha }.$

\section{Background\label{background}}

We begin with a brief review of the few properties of the density topology
on $\Bbb{R}^{N}$ and related topics that will be used in this paper. Further
information, notably the proof that the density topology \emph{is} a
topology, can be found in \cite{GoNeNi61} or \cite{LuMaZa86}. It was
introduced in 1952 by Haupt and Pauc \cite{HaPa52} in a more general
setting, but many other expositions are limited to $N=1.$ For classical
generalizations, see \cite{Ma64}, \cite{TrZi63}.

First, recall that while the density of a set at a point $x$ is often
defined by using shrinking families of open cubes centered at $x,$ an
equivalent definition is obtained if cubes are replaced with euclidean
balls. This is elementary but still requires a short argument; see for
instance \cite[p. 460]{Jo93}. While not a major point, this remark is
convenient.

From the very definition of a density open subset, it follows that the
density interior of a \emph{measurable} subset $S\subset \Bbb{R}^{N}$ is the
subset $S_{1}$ of $S$ of those points at which $S$ has density $1.$ By the
Lebesgue density theorem, $S\backslash S_{1}$ is a null set. Thus, a null
set has empty density interior and, conversely, a measurable set with empty
density interior is a null set. (This converse is of course false with the
euclidean topology.) In particular, a nonempty density open subset always
has positive measure.

Every subset of $\Bbb{R}^{N},$ measurable or not, has a density interior,
but a non-measurable subset with empty density interior is obviously not a
null set. Such sets will never be involved in the sequel. Although we shall
not use this here, we feel compelled to point out that every null set is
density closed (and even discrete) because its complement is clearly density
open.

A measurable subset $W\subset \Bbb{R}^{N}$ is a density neighborhood of a
point $x$ if and only if it contains a density open neighborhood of $x.$
From the above, this happens if and only if $W$ has density $1$ at $x$ and
then $W$ has positive measure. Thus, the inverse image $f^{-1}(V)$ of an
open subset $V\subset \Bbb{R}$ under a \emph{measurable} function $f$ is a
density neighborhood of some point $x$ if and only if $f^{-1}(V)$ has
density $1$ at $x.$

In the Introduction, a function $f:\Bbb{R}^{N}\rightarrow \Bbb{R}$ was
called approximately continuous if it is continuous when $\Bbb{R}^{N}$ is
equipped with the density topology and $\Bbb{R}$ with the euclidean
topology. A different definition is that every $x\in \Bbb{R}^{N}$ is
contained in a measurable set $E_{x}$ having density $1$ at $x$ such that $%
f_{|E_{x}}$ is continuous at $x$ (for the euclidean topology). It is well
known and not hard to prove, though not entirely trivial, that the two
definitions are equivalent.

Aside from the density topology and approximately continuous functions, we
shall also use several properties of convex subsets of $\Bbb{R}^{N},$ some
of which, but not all, are explicitly spelled out in standard texts. A basic
fact is that if a convex subset $C\subset \Bbb{R}^{N}$ has empty interior,
it is contained in an affine hyperplane (\cite{R070}). Then, elementary
considerations yield the following: For every convex subset $C\subset \Bbb{R}%
^{N}$ the statements (i) $C$ has empty interior, (ii) $C$ is a null set
(iii) $\overline{C}$ is a null set and (iv) $\overline{C}$ has empty
interior, are all equivalent.

Another useful property is that if $C\subset \Bbb{R}^{N}$ is closed and
convex, at least one supporting hyperplane passes through each point of its
boundary $\partial C.$ Furthermore, every convex subset $C\subset \Bbb{R}
^{N} $ is measurable because $C$ is the union of its interior $\overset{
\circ }{C} $ with a subset of $\partial C,$ and $\partial C$ is always a
null set.

The last statement is one of those folklore results routinely used without
being linked to a reference, or linked to one with a technical proof. There
must certainly be counter examples to this statement, but supporting
evidence is far more common. So, here is a quick proof for convenience and
completeness. Since $\partial C$ is closed, it is measurable. In the only
nontrivial case when $C$ has nonempty interior, every open ball $B(x,r)$
centered at $x\in \partial C$ is split into two open halves by a hyperplane
supporting $\overline{C}$ at $x,$ so that one of them does not intersect $%
\overline{C}.$ As a result, the density of $\partial C$ at $x$ cannot exceed 
$\frac{1}{2},$ so that $\partial C$ is a null set by the Lebesgue density
theorem (accordingly, the density of $\partial C$ at any of its points is
actually $0$). Incidentally, the ball-splitting argument will soon be used
again for other purposes.

Notice that the measurability of convex sets implies at once that all
quasiconvex functions are measurable.

\section{Main result\label{main}}

We need two preliminary lemmas.

\begin{lemma}
\label{lm1}Let $C\subset \Bbb{R}^{N}$ be convex and $U\subset \Bbb{R}^{N}$
be density open.\newline
(i) If $C$ has nonempty interior and $U\cap C\neq \emptyset ,$ then $\mu
_{N}(U\cap C)>0.$\newline
(ii) If $U\cap (\Bbb{R}^{N}\backslash C)\neq \emptyset ,$ then $\mu
_{N}(U\cap (\Bbb{R}^{N}\backslash C))>0.$
\end{lemma}

\begin{proof}
(i) Choose $x_{0}\in U\cap C$ along with an open ball $B\subset C$ such that 
$x_{0}\notin B.$ The hypothesis $\overset{\circ }{C}\neq \emptyset $ ensures
that $B$ exists. Indeed, choose $B\subset \overset{\circ }{C}.$ If $x_{0}$
is not the center of $B,$ shrink the radius of $B$ until $x\notin B.$ If $%
x_{0}$ is the center of $B,$ just replace $B$ by an open ball contained in $%
B $ that does not contain the center $x_{0}.$

The set $K:=\cup _{\lambda \in (0,1)}\lambda B+(1-\lambda )x_{0}$ is an open
convex cone with apex at $x_{0}$ (and spherical ``end'') contained in $%
\overset{\circ }{C}.$ Let $B(x_{0},r)$ denote the open ball with center $%
x_{0}$ and radius $r>0.$ Clearly, the ratio $\kappa :=\frac{\mu _{N}(K\cap
B(x_{0},r))}{\mu _{N}(B(x_{0},r))}\in (0,1)$ is independent of $r>0$ small
enough. On the other hand, since $U$ has density $1$ at $x_{0},$ then $\frac{%
\mu _{N}(U\cap B(x_{0},r))}{\mu _{N}(B(x_{0},r))}>1-\kappa $ if $r>0$ is
small enough. This implies $\mu _{N}(U\cap K\cap B(x_{0},r))>0,$ for
otherwise the intersection of $U\cap B(x_{0},r)$ and $K\cap B(x_{0},r)$
(that is, $U\cap K\cap B(x_{0},r)$) is a null set, so that $\mu _{N}((U\cup
K)\cap B(x_{0},r))=\mu _{N}(U\cap B(x_{0},r))+\mu _{N}(K\cap B(x_{0},r))>\mu
_{N}(B(x_{0},r)),$ which is absurd. Since $K\cap B(x_{0},r)\subset K\subset
C,$ it follows that $\mu _{N}(U\cap C)>0.$

(ii) Choose $x_{0}\in U\cap (\Bbb{R}^{N}\backslash C).$ We claim that $\Bbb{R%
}^{N}\backslash C$ contains (at least) half of any open ball centered at $
x_{0}$ with small enough radius. Since this is obvious if $x_{0}$ lies in
the interior of $\Bbb{R}^{N}\backslash C,$ we assume that $x_{0}\in \partial
(\Bbb{R}^{N}\backslash C)=\partial C.$ There is at least one affine
hyperplane $H$ supporting $\overline{C}$ at $x_{0}.$ Therefore, $H$ splits
every open ball $B(x_{0},r)$ into two open halves, one of which does not
intersect $\overline{C}$ and is therefore contained in $\Bbb{R}
^{N}\backslash C.$

Since $U$ has density $1$ at $x_{0},$ it follows that $\mu _{N}(U\cap
B(x_{0},r))>\frac{1}{2}\mu _{N}(B(x_{0},r))$ if $r>0$ is small enough. From
the above, half of $B(x_{0},r)$ is contained in $\Bbb{R}^{N}\backslash C$
and the other half cannot contain a set of measure greater than $\frac{1}{2}%
\mu _{N}(B(x_{0},r)).$ As a result, the half-ball contained in $\Bbb{R}%
^{N}\backslash C$ must intersect $U$ along a set of positive measure, so
that $\mu _{N}(U\cap (\Bbb{R}^{N}\backslash C))>0.$
\end{proof}

\begin{lemma}
\label{lm2}If $f:\Bbb{R}^{N}\rightarrow \Bbb{R}$ is measurable, the
following statements are equivalent:\newline
(i) $\sup_{U}f=\limfunc{ess}\sup_{U}f$ for every density open subset $%
U\subset \Bbb{R}^{N}.$ \newline
(ii) For every $x_{0}\in f^{-1}(\Bbb{R}),$ every density open subset $%
U\subset \Bbb{R}^{N}$ containing $x_{0}$ and every $\varepsilon >0,$%
\begin{equation*}
\mu _{N}(\{x\in U:f(x)\geq f(x_{0})-\varepsilon \})>0.
\end{equation*}
Likewise, the following statements are equivalent: \newline
(i') $\inf_{U}f=\limfunc{ess}\inf_{U}f$ for every density open subset $%
U\subset \Bbb{R}^{N}.$ \newline
(ii') For every $x_{0}\in f^{-1}(\Bbb{R}),$ every density open subset $%
U\subset \Bbb{R}^{N}$ containing $x_{0}$ and every $\varepsilon >0,$%
\begin{equation*}
\mu _{N}(\{x\in U:f(x)<f(x_{0})+\varepsilon \})>0.
\end{equation*}
\end{lemma}

\begin{proof}
(i) $\Rightarrow $ (ii) Suppose that (i) holds and, by contradiction, assume
that there are $x_{0}\in f^{-1}(\Bbb{R}),$ a density open subset $U\subset 
\Bbb{R}^{N}$ and some $\varepsilon >0$ such that $\mu _{N}(\{x\in U:f(x)\geq
f(x_{0})-\varepsilon \})=0.$ Then, $\limfunc{ess}\sup_{U}f\leq
f(x_{0})-\varepsilon <f(x_{0})\leq \sup_{U}f,$ which contradicts (i).

(ii) $\Rightarrow $ (i) Let $U\subset $ $\Bbb{R}^{N}$ be a density open
subset. We argue by contradiction, thereby assuming that $\sup_{U}f>\limfunc{
ess}\sup_{U}f$ . If so, $U$ is not empty (otherwise, both suprema are $%
-\infty $ ) and $\limfunc{ess}\sup_{U}f<\infty .$ Thus, the assumption $%
\sup_{U}f>\limfunc{ess}\sup_{U}f$ implies the existence of $x_{0}\in f^{-1}(%
\Bbb{R})$ such that $\limfunc{ess}\sup_{U}f<f(x_{0})\leq \sup_{U}f.$ Choose $%
\varepsilon >0$ small enough that $\limfunc{ess}\sup_{U}f<f(x_{0})-%
\varepsilon .$ By (ii), $\mu _{N}(\{x\in U:f(x)\geq f(x_{0})-\varepsilon
\})>0,$ so that $\limfunc{ess}\sup_{U}f\geq f(x_{0})-\varepsilon ,$ which is
a contradiction.

That (i') $\Leftrightarrow $ (ii') follows by replacing $f$ by $-f$ above,
after noticing that the equivalence between (i) and (ii) remains true if the
inequality in $\{x\in U:f(x)\geq f(x_{0})-\varepsilon \}$ is replaced by the
corresponding strict inequality.
\end{proof}

We now prove the main result announced in the Introduction.

\begin{theorem}
\label{th3}Let $f:\Bbb{R}^{N}\rightarrow \overline{\Bbb{R}}$ be quasiconvex.%
\newline
(i) $\sup_{U}f=\limfunc{ess}\sup_{U}f$ for every density open subset $%
U\subset \Bbb{R}^{N}.$\newline
(ii) $\inf_{U}f=\limfunc{ess}\inf_{U}f$ for every density open subset $%
U\subset \Bbb{R}^{N}$ if and only if this is true when $U=\Bbb{R}^{N}.$
\end{theorem}

\begin{proof}
The extended real-valued case can be deduced from the real-valued one by
changing $f$ into $\arctan f.$ This does not affect quasiconvexity and it is
easily checked that $\arctan $ commutes with $\limfunc{ess}\sup_{U}$ and $%
\limfunc{ess}\inf_{U}.$ Accordingly, in the remainder of the proof, $f$ is
real-valued.

(i) We show that the condition (ii) of Lemma \ref{lm2} holds and use the
equivalence with (i) of that lemma.

Pick $x_{0}\in f^{-1}(\Bbb{R}),$ a density open subset $U\subset \Bbb{R}^{N}$
containing $x_{0}$ and $\varepsilon >0.$ The set $\{x\in U:f(x)\geq
f(x_{0})-\varepsilon \}$ is the intersection $U\cap (\Bbb{R}^{N}\backslash
F_{f(x_{0})-\varepsilon })$ (see (\ref{1})). Since $U\cap (\Bbb{R}
^{N}\backslash F_{f(x_{0})-\varepsilon })\neq \emptyset $ (it contains $%
x_{0} $), it follows from part (ii) of Lemma \ref{lm1} that $\mu _{N}(U\cap (%
\Bbb{R}^{N}\backslash F_{f(x_{0})-\varepsilon }))>0.$

(ii) It is obvious that $\inf_{\Bbb{R}^{N}}f=\limfunc{ess}\inf_{\Bbb{R}%
^{N}}f $ is necessary. Conversely, assuming this, we show that the condition
(ii') of Lemma \ref{lm2} holds and use the equivalence with (i') of that
lemma.

Pick $x_{0}\in f^{-1}(\Bbb{R}),$ a density open subset $U\subset \Bbb{R}^{N}$
containing $x_{0}$ and $\varepsilon >0.$ The set $\{x\in
U:f(x)<f(x_{0})+\varepsilon \}$ is the intersection $U\cap
F_{f(x_{0})+\varepsilon }.$ Since $\limfunc{ess}\inf_{\Bbb{R}^{N}}f=\inf_{%
\Bbb{R}^{N}}f\leq f(x_{0})<f(x_{0})+\varepsilon ,$ the set $%
F_{f(x_{0})+\varepsilon }$ has positive measure and hence nonempty interior
since it is convex. Therefore, $\mu _{N}(U\cap F_{f(x_{0})+\varepsilon })>0$
by part (i) of Lemma \ref{lm1}.
\end{proof}

For convex functions (defined as functions with convex epigraphs and hence
not necessarily proper), Theorem \ref{th3} can be phrased differently.
Recall that the domain $\limfunc{dom}f$ of a convex function $f$ is the set
of points where $f<\infty .$ It includes the points where $f=-\infty ,$ if
any.

\begin{corollary}
\label{cor4}Let $f:\Bbb{R}^{N}\rightarrow [-\infty ,\infty ]$ be convex.
Then:\newline
(i) $\sup_{U}f=\limfunc{ess}\sup_{U}f$ for every density open subset $
U\subset \Bbb{R}^{N}.$\newline
(ii) $\inf_{U}f=\limfunc{ess}\inf_{U}f$ for every density open subset $%
U\subset \Bbb{R}^{N}$ if and only if either $f=\infty $ everywhere, or $%
\limfunc{dom}f$ has nonempty interior.
\end{corollary}

\begin{proof}
Since there is no need to discuss the case when $f=\infty $ everywhere
(trivial convex function), we henceforth assume that $f$ is not trivial. By
Theorem \ref{th3}, it suffices to prove that $\inf_{\Bbb{R}^{N}}f=\limfunc{
ess}\inf_{\Bbb{R}^{N}}f$ if and only if $\limfunc{dom}f$ has nonempty
interior.

We begin with necessity: If $\inf_{\Bbb{R}^{N}}f=\limfunc{ess}\inf_{\Bbb{R}
^{N}}f$ and $\limfunc{dom}f\neq \emptyset ,$ then $\limfunc{dom}f$ has
nonempty interior, for otherwise $\limfunc{dom}f$ (convex) is a null set, so
that $\limfunc{ess}\inf_{\Bbb{R}^{N}}f=\infty $ while $\inf_{\Bbb{R}
^{N}}f<\infty $ since $f$ is not trivial.

The proof of sufficiency requires a little work. For convenience, we set $D:=%
\limfunc{dom}f$ and, from now on, assume $\overset{\circ }{D}\neq \emptyset
. $

That $\inf_{\Bbb{R}^{N}}f=\limfunc{ess}\inf_{\Bbb{R}^{N}}f=-\infty $ is
trivial if $f^{-1}(-\infty )$ is a (convex) set of positive measure. Thus,
it suffices to consider the case when $f^{-1}(-\infty )$ is a null set. If
so, $f^{-1}(-\infty )=\emptyset .$ To see this, assume by contradiction that 
$x\in f^{-1}(-\infty ).$ Since $f^{-1}(-\infty )$ is convex and a null set,
its closure $\overline{f^{-1}(-\infty )}$ is also a null set (Section \ref
{background}). Thus, $\overset{\circ }{D}\neq \emptyset $ ensures that $%
\overset{\circ }{D}\backslash \overline{f^{-1}(-\infty )}$ contains an open
ball $B.$ Call $x_{0}$ its center and note that $f(x_{0})\in \Bbb{R}.$

By the convexity of $f$ (\cite[p. 25]{R070}), $f(\lambda x+(1-\lambda
)x_{0})<\lambda \alpha +(1-\lambda )\beta $ for every $\lambda \in (0,1)$
and every $\alpha ,\beta \in \Bbb{R}$ such that $\alpha >f(x)$ and $\beta
>f(x_{0}).$ Since $f(x)=-\infty ,$ it follows that $f(\lambda x+(1-\lambda
)x_{0})=-\infty $ for every $\lambda \in (0,1].$ But if $\lambda >0$ is
small enough, then $\lambda x+(1-\lambda )y\in B,$ so that $f(\lambda
x+(1-\lambda )y)\in \Bbb{R}.$ This contradiction shows that $f^{-1}(-\infty
)=\emptyset $ and, hence, that $f$ is proper (i.e., finite on $D$).

Since $f=\infty $ outside $D$ and $D$ has positive measure, $\limfunc{ess}%
\inf_{\Bbb{R}^{N}}f=\limfunc{ess}\inf_{D}f.$ Also, it is plain that $\inf_{%
\Bbb{R}^{N}}f=\inf_{D}f.$ To complete the proof, it suffices to show that $%
f(x)\geq \limfunc{ess}\inf_{D}f$ for every $x\in D.$ This is obvious if $%
x\in \overset{\circ }{D}$ since a proper convex function is continuous on
the interior of its domain.

Let then $x\in D\cap \partial D.$ Given $y\in \overset{\circ }{D},$ the
segment $(x,y]$ is entirely contained in $\overset{\circ }{D}$ (\cite[p. 45]
{R070}). On the other hand, a (finite) convex function on an interval is
upper semicontinuous on the closure of that interval (\cite[p. 16]{HiLe96}).
Thus, $f(x)\geq \lim \sup_{z\rightarrow x,z\in (x,y]}f(z).$ As observed
earlier, $f(z)\geq \limfunc{ess}\inf_{D}f$ since $z\in \overset{\circ }{D},$
so that $f(x)\geq \limfunc{ess}\inf_{D}f.$
\end{proof}

\section{Common points of continuity of equivalent quasiconvex functions%
\label{equivalent}}

The equivalence referred to in the section head is equality a.e. A point of
(approximate) continuity of $f:\Bbb{R}^{N}\rightarrow \overline{\Bbb{R}}$ is
defined as a point $x\in f^{-1}(\Bbb{R})$ such that $f$ is (approximately)
continuous at $x.$ Such points are the points $x$ of (approximate)
continuity of $\arctan f$ such that $\arctan f(x)\neq \pm \frac{\pi }{2}.$
Thus, as in the proof of Theorem \ref{th3}, we may and will confine
attention to real-valued functions.

\begin{lemma}
\label{lm5}Let $f,g:\Bbb{R}^{N}\rightarrow \Bbb{R}$ be quasiconvex functions
such that $f=g$ a.e., so that $\limfunc{ess}\inf_{\Bbb{R}^{N}}f=\limfunc{ess}%
\inf_{\Bbb{R}^{N}}g:=m$ ($\geq -\infty $).\newline
(i) If also $\inf_{\Bbb{R}^{N}}f=m$ and $\inf_{\Bbb{R}^{N}}g=m$ (not a
restriction if $m=-\infty $), then $f$ and $g$ have the same points of
continuity, the same points of approximate continuity and achieve a common
value at such points. \newline
(ii) $\max \{f,m\}$ and $\max \{g,m\}$ have the same points of continuity,
the same points of approximate continuity and achieve a common value at such
points.
\end{lemma}

\begin{proof}
(i) Let $x$ denote a point of continuity of $f,$ so that for every $%
\varepsilon >0,$ there is an open neighborhood $U$ of $x$ such that $%
f(U)\subset [f(x)-\varepsilon ,f(x)+\varepsilon ].$ Thus, $\inf_{U}f\geq
f(x)-\varepsilon $ and $\sup_{U}f\leq f(x)+\varepsilon .$ By parts (i) and
(ii) of Theorem \ref{th3}, this is the same as $\limfunc{ess}\inf_{U}f\geq
f(x)-\varepsilon $ and $\limfunc{ess}\sup_{U}f\leq f(x)+\varepsilon .$

Since $f=g$ a.e., the essential extrema are unchanged when $f$ is replaced
by $g,$ so that $\limfunc{ess}\inf_{U}g\geq f(x)-\varepsilon $ and $\limfunc{
ess}\sup_{U}g\leq f(x)+\varepsilon .$ By using once again parts (i) and (ii)
of Theorem \ref{th3}, it follows that $\inf_{U}g\geq f(x)-\varepsilon $ and $%
\sup_{U}g\leq f(x)+\varepsilon ,$ whence $g(U)\subset [f(x)-\varepsilon
,f(x)+\varepsilon ].$ In particular, $g(x)\in [f(x)-\varepsilon
,f(x)+\varepsilon ].$ Since $\varepsilon >0$ is arbitrary, it follows that $%
g(x)=f(x)$ and hence that $g(U)\subset [g(x)-\varepsilon ,g(x)+\varepsilon
], $ which proves the continuity of $g$ at $x.$

In summary, the points of continuity of $f$ are points of continuity of $g$
and $g=f$ at such points. By exchanging the roles of $f$ and $g,$ the
converse is true.

The exact same argument as above can be repeated for the points of
approximate continuity since Theorem \ref{th3} is applicable when $U$ is
density open.

(ii) Just use (i) with $\max \{f,m\}$ and $\max \{g,m\},$ respectively.
Neither quasiconvexity nor a.e. equality is affected and $\limfunc{ess}\inf_{%
\Bbb{R}^{N}}\max \{f,m\}=\limfunc{ess}\inf_{\Bbb{R}^{N}}\max \{g,m\}=m,$ so
that $\inf_{\Bbb{R}^{N}}\max \{f,m\}=m=\inf_{\Bbb{R}^{N}}\max \{g,m\}$ is
obvious.
\end{proof}

Part (ii) of Lemma \ref{lm5} will now be instrumental to identify simple
necessary and sufficient conditions ensuring that a given point of
(approximate) continuity of one function, say $g,$ is not a point of
(approximate) continuity of $f:$

\begin{theorem}
\label{th6}Let $f,g:\Bbb{R}^{N}\rightarrow \Bbb{R}$ be quasiconvex functions
such that $f=g$ a.e., so that $\limfunc{ess}\inf_{\Bbb{R}^{N}}f=\limfunc{ess}%
\inf_{\Bbb{R}^{N}}g:=m$ ($\geq -\infty $).\newline
(i) If $x\in \Bbb{R}^{N}$ is a point of approximate continuity of $g,$ then $%
g(x)\geq m.$ Furthermore, $x$ is a point of approximate continuity of $g,$
but not one of $f,$ if and only if $m>-\infty ,g(x)=m$ and $x\in F_{m},$ a
set of measure $0$.\newline
(ii) If $x\in \Bbb{R}^{N}$ is a point of continuity of $g,$ then $g(x)\geq
m. $ Furthermore, $x$ is a point of continuity of $g,$ but not one of $f,$
if and only if $m>-\infty ,g(x)=m$ and $x\in \cup _{\alpha <m}\overline{F}%
_{\alpha }\subset \overline{F}_{m},$ a set of measure $0.$ \newline
\end{theorem}

\begin{proof}
With no loss of generality, assume $m>-\infty $ since, otherwise, everything
follows at once from part (i) of Lemma \ref{lm5}. We first justify the
statement that $F_{m}$ and $\cup _{\alpha <m}\overline{F}_{\alpha }$ are
null sets. Notice that $F_{m}=\cup _{\alpha <m,\alpha \in \Bbb{Q}}F_{\alpha
} $ and that each $F_{\alpha }$ with $\alpha <m$ is a null set by definition
of $m.$ Thus, $F_{m}$ is a null set and therefore $\overline{F}_{m}$ is also
a null set since $F_{m}$ is convex (see Section \ref{background}). Thus, $%
\cup _{\alpha <m}\overline{F}_{\alpha }\subset \overline{F}_{m}$ is a null
set.

(i) By contradiction, assume that $x$ is a point of approximate continuity
of $g$ and that $g(x)<m.$ Pick $\alpha \in \Bbb{R}$ such that $g(x)<\alpha
<m.$ By definition of $m,$ the set $G_{\alpha }:=\{y\in \Bbb{R}%
^{N}:g(y)<\alpha \}$ is a null set. On the other hand, since $x\in G_{\alpha
}=g^{-1}((-\infty ,\alpha )),$ the approximate continuity of $g$ at $x$
implies that $G_{\alpha }$ is a density neighborhood of $x,$ so that it has
positive measure. This contradiction proves that $g(x)\geq m,$ as claimed.

Next, let $x$ be a point of approximate continuity of $g$ and hence one of $%
\max \{g,m\}$. By part (ii) of Lemma \ref{lm5}, $x$ is a point of
approximate continuity of $\max \{f,m\}$ and $\max \{g(x),m\}=\max
\{f(x),m\}.$ Therefore, if $g(x)>m$ \emph{or} $f(x)>m,$ then $g(x)>m$ \emph{%
\ \ and} $f(x)>m.$ To see that $x$ is a point of approximate continuity of $%
f,$ choose $\varepsilon >0$ small enough that $m<f(x)-\varepsilon $ and let $%
I_{\varepsilon }:=(f(x)-\varepsilon ,f(x)+\varepsilon ).$ Then, $(\max
\{f,m\})^{-1}(I_{\varepsilon })$ is a density neighborhood $W_{\varepsilon }$
of $x.$ From the choice of $\varepsilon ,$ it is obvious that $%
W_{\varepsilon }=f^{-1}(I_{\varepsilon })$ . Since this is true for every $%
\varepsilon >0$ small enough, it follows that $f$ is approximately
continuous at $x.$

From the above, if $x$ is a point of approximate continuity of $g,$ but not
one of $f,$ then $g(x)=m$ and $f(x)\leq m.$ As was seen earlier (with $g$
instead of $f$), $x$ is not a point of approximate continuity of $f$ if $%
f(x)<m.$ It remains to prove that the converse is true, i.e., that if $%
f(x)=m,$ then $f$ is approximately continuous at $x.$

It suffices to show that if $\alpha <m<\beta $ and $I:=(\alpha ,\beta ),$
then $f^{-1}(I)$ is a density neighborhood of $x,$ i.e., that $f^{-1}(I)$
has density $1$ at $x$ (since $f^{-1}(I)$ is measurable). Now, $\max \{f,m\}$
\emph{is} approximately continuous at $x,$ whence $(\max \{f,m\})^{-1}(I)$
does have density $1$ at $x.$ Since $m<\beta ,$ we may split $(\max
\{f,m\})^{-1}(I)=F_{m}\cup E$ with $E:=\{y\in \Bbb{R}^{N}:m\leq f(y)<\beta
\} $ and we already know that $F_{m}$ is a null set. Therefore, $E$ and $%
(\max \{f,m\})^{-1}(I)$ have the same density at every point of $\Bbb{R}
^{N}. $ In particular, $E$ has density $1$ at $x,$ so that $f^{-1}(I)\supset
E$ has density $1$ at $x,$ as claimed.

(ii) That $g(x)\geq m$ follows at once from (i). The proof that $x$ is a
point of continuity of $f$ if it is one of $g$ and either $g(x)>m$ or $%
f(x)>m $ proceeds as above, by merely changing the terminology in the
obvious way. Thus, it only remains to show that if $x$ is a point of
continuity of $g$ such that $g(x)=m$ and $f(x)\leq m,$ it is not a point of
continuity of $f$ if and only if $x\in \cup _{\alpha <m}\overline{F}_{\alpha
}.$

If $f(x)<m,$ i.e., $x\in F_{m},$ then by (i) $x$ is not a point of
approximate continuity of $f,$ so it is not a point of continuity of $f$ and 
$x\in \cup _{\alpha <m}\overline{F}_{\alpha }$ since $F_{m}=\cup _{\alpha
<m}F_{\alpha }.$ The only thing left to prove that if $f(x)=m,$ then $x$ is
not a point of continuity of $f$ if and only if $x\in \cup _{\alpha <m}%
\overline{F}_{\alpha }.$

Suppose first that $x\in \cup _{\alpha <m}\overline{F}_{\alpha },$ so that
there is $\alpha <m$ such that $x\in \overline{F}_{\alpha }.$ As a result,
there is a sequence $(x_{n})\subset F_{\alpha }$ tending to $x.$ But since $%
f(x_{n})<\alpha <m,$ it is obvious that $f(x_{n})$ does not tend to $f(x)=m$%
. This proves that $x$ is not a point of continuity of $f.$

Conversely, still with $f(x)=m,$ we claim that if $x\notin \cup _{\alpha <m}%
\overline{F}_{\alpha },$ then $x$ is a point of continuity of $f.$ Let $%
(x_{n})$ be a sequence tending to $x$ and let $\varepsilon >0$ be given. If $%
n$ is large enough, then $f(x_{n})\geq m-\varepsilon ,$ for otherwise there
is a subsequence $(x_{n_{k}})$ such that $x_{n_{k}}\in F_{m-\varepsilon },$
so that $x\in \overline{F}_{m-\varepsilon },$ which is not the case.
Therefore, $\lim \inf_{n\rightarrow \infty }f(x_{n})\geq m-\varepsilon .$

Since $x$ is a point of continuity of $g,$ it is one of $\max \{g,m\}.$
Thus, from part (ii) of Lemma \ref{lm5}, $\max \{f,m\}$ is continuous at $x$
and so $\lim_{n\rightarrow \infty }\max \{f(x_{n}),m\}=\max \{f(x),m\}=m$
since $f(x)=m.$ It follows that $\lim \sup_{n\rightarrow \infty
}f(x_{n})\leq m.$ In summary, $m-\varepsilon \leq \lim \inf_{n\rightarrow
\infty }f(x_{n})\leq \lim \sup_{n\rightarrow \infty }f(x_{n})\leq m.$ Since $%
\varepsilon >0$ is arbitrary, $\lim \inf_{n\rightarrow \infty }f(x_{n})=\lim
\sup_{n\rightarrow \infty }f(x_{n})=m=f(x),$ which proves that $f$ is
continuous at $x.$
\end{proof}

For completeness, we give an example when $\cup _{\alpha <m}\overline{F}
_{\alpha }\neq \overline{F}_{m}.$

\begin{example}
In $\Bbb{R}^{2}$ with $x=(x_{1},x_{2}),$ let $f(x)=|x_{1}|$ if $x_{2}\geq 0$
or if $x_{2}<0,x_{1}\neq 0$ and let $f(0,x_{2})=x_{2}$ if $x_{2}<0.$ Then $f$
is quasiconvex, $m=0$ and $\cup _{\alpha <0}\overline{F}_{\alpha
}=\{0\}\times (-\infty ,0),$ but $\overline{F}_{0}=\{0\}\times (-\infty ,0].$
Observe that $f$ is continuous at $(0,0).$ This is no longer true if $f$ is
modified by setting $f(0,x_{2})=-1$ if $x_{2}<0,$ but $f$ is still
approximately continuous at $(0,0).$
\end{example}

The next corollary generalizes part (i) of Lemma \ref{lm5}. The proof is
mostly a rephrasing of Theorem \ref{th6}. The only extra technicality is to
show that if $f$ and $g$ are (approximately) continuous at the same point $%
x, $ they must coincide at that point. Since $f=g$ a.e., this is obvious,
but we spell out the argument in the approximately continuous case: If $%
f(x)\neq g(x),$ there is a density neighborhood $W$ of $x$ such that $%
f(y)\neq g(y)$ for every $y\in W.$ Since $W$ has positive measure (Section 
\ref{background}), a contradiction arises with $f=g$ a.e.

\begin{corollary}
\label{cor7}Let $f,g:\Bbb{R}^{N}\rightarrow \Bbb{R}$ be quasiconvex
functions such that $f=g$ a.e., so that $\limfunc{ess}\inf_{\Bbb{R}^{N}}f=%
\limfunc{ess}\inf_{\Bbb{R}^{N}}g:=m$ ($\geq -\infty $).\newline
(i) Every point of approximate continuity of $g$ is a point of approximate
continuity of $f$ if and only if $g$ has no point of approximate continuity $%
x$ such that $g(x)=m$ and $f(x)<m$ (always true if $m=-\infty $). If so, $%
f(x)=g(x)$ at every point $x$ of approximate continuity of $g.$\newline
(ii) Every point of continuity of $g$ is a point of continuity of $f$ if and
only if $g$ has no point of continuity $x$ such that $g(x)=m$ and $%
x=\lim_{n\rightarrow \infty }x_{n}$ where $(x_{n})$ is a sequence such that $%
f(x_{n})<\alpha <m$ for some $\alpha \in \Bbb{R}$ and every $n\in \Bbb{N}$
(always true if $m=-\infty $). If so, $f(x)=g(x)$ for every point of
continuity $x$ of $g.$
\end{corollary}

Clearly, $\inf_{\Bbb{R}^{N}}f=m\geq -\infty $ is a only an especially simple
special case when the conditions given in (i) and (ii) of Corollary \ref
{cor7} hold. If also $\inf_{\Bbb{R}^{N}}g=m,$ the roles of $f$ and $g$ can
be exchanged in Corollary \ref{cor7}, so that $f$ and $g$ have the same
points of continuity and part (i) of Lemma \ref{lm5} is recovered.

\end{document}